\documentclass[a4paper,12pt]{article}
\usepackage{amsmath,amssymb}
\usepackage{parskip,url}
\usepackage{cancel}
\newtheorem{theorem}{Theorem}[section]
\newtheorem{lemma}[theorem]{Lemma}
\newtheorem{proposition}[theorem]{Proposition}

\newcommand{\ch}{\operatorname{char}}

\newcommand{\id}{\mathrm{id}}
\newcommand{\supp}{\operatorname{supp}}
\newcommand{\gr}{\operatorname{gr}}

\newcommand{\Sym}{\operatorname{Sym}}
\newcommand{\FM}{\operatorname{FM}}

\newcommand{\GKdim}{\operatorname{GKdim}}
\newcommand{\clKdim}{\operatorname{clKdim}}
\newcommand{\tr}{\operatorname{tr}}

\newenvironment{proof}{\par\noindent{\bf Proof.}}{$\qed$\par\bigskip}
\newcommand{\qed}{\enspace\vrule  height6pt  width4pt  depth2pt}
\begin{document}

\title{Finitely Presented Monoids and Algebras defined by Permutation Relations of Abelian Type, II}

\author{ Ferran Ced\'{o}\footnote{
 Research partially supported by grants of MICIIN (Spain)
MTM2011-28992-C02-01, Generalitat de Catalunya 2009 SGR 1389.} \and
Eric Jespers\footnote{Research supported in part by Onderzoeksraad
of Vrije Universiteit Brussel and Fonds voor Wetenschappelijk
Onderzoek (Belgium).} \and Georg Klein}
\date{}
\maketitle

\begin{abstract}
The class of finitely presented algebras $A$ over a field $K$ with a
set of generators $x_{1},\ldots ,x_{n}$ and defined by homogeneous
relations of the form
 $x_{i_1}x_{i_2}\cdots x_{i_l} =x_{\sigma (i_1)} x_{\sigma (i_2)} \cdots
 x_{\sigma (i_l)}$,
where $l\geq 2$ is a given  integer and  $\sigma$ runs through a subgroup $H$ of $\Sym_{n}$, is
considered. It is shown that the underlying monoid  $S_{n,l}(H)=\langle x_1,x_2,\dots ,x_n\mid
x_{i_1}x_{i_2}\cdots x_{i_l} =x_{\sigma (i_1)} x_{\sigma (i_2)}
\cdots
 x_{\sigma (i_l)}, \;
 \sigma \in H, \;   i_1,\ldots ,i_l \in  \{ 1, \ldots , n \}
\rangle $ is cancellative if and only if $H$ is semiregular and abelian. In this case $S_{n,l}(H)$ is a submonoid of its universal group $G$.
If, furthermore, $H$ is transitive then the periodic elements $T(G)$ of $G$ form a finite  abelian subgroup, $G$ is periodic-by-cyclic and it is a central localization of $S_{n,l}(H)$, and
the Jacobson radical of the algebra $A$ is determined by the Jacobson radical of the group algebra $K[T(G)]$.
Finally, it is shown that if $H$ is an arbitrary group that is transitive then $K[S_{n,l}(H)]$ is a Noetherian PI-algebra of Gelfand-Kirillov dimension one; if furthermore $H$ is abelian then often $K[G]$ is a principal ideal ring. In case $H$ is not transitive then $K[S_{n,l}(H)]$ is of exponential growth.
\end{abstract}


\section{Introduction}

Many  concrete classes of finitely presented algebras defined by
homogeneous relations receive substantial interest in the
literature. For example, Chinese and plactic algebras and algebras
arising from set theoretic solutions of the Yang-Baxter equation.
Motivated by these examples,  the authors and Okni\'{n}ski in
\cite{alghomshort,altalgebra,alt4algebra,symcyclic} investigated
algebras $A$ over a field $K$ that are of the following type:
$$A=K\langle a_1,a_2,\dots ,a_n\mid a_1a_2\cdots
a_n= a_{\sigma (1)}a_{\sigma (2)}\cdots a_{\sigma (n)},\; \sigma\in
H\rangle , $$ where $H$ is a subgroup of the symmetric group
$\Sym_n$ of degree $n$.   Results were obtained mainly in case $H$
is the full symmetric group \cite{alghomshort}, the alternating
group of degree $n$ \cite{altalgebra,alt4algebra}, or an abelian
group \cite{symcyclic}.  Clearly such an algebra $A$ is a monoid
algebra $K[S_{n}(H)$] of the monoid $S_n(H)=\langle a_1,a_2,\dots
,a_n\mid a_1a_2\cdots a_n= a_{\sigma (1)}a_{\sigma (2)}\cdots
a_{\sigma (n)},\; \sigma\in H\rangle $. Hence, in all cases studied,
a thorough investigation of the monoid $S_{n}(H)$ is crucial and
leads to a structural description of the semigroup algebra
$K[S_n(H)]$.

The aim of this paper is to widen these investigations to the  new
class of   finitely presented algebras of the type:
$$A=K\langle x_1,x_2,\dots ,x_n\mid
x_{i_1}x_{i_2}\cdots x_{i_l} =x_{\sigma (i_1)} x_{\sigma (i_2)}
\cdots
 x_{\sigma (i_l)},\; \sigma\in
H\rangle , $$ where $l\geq 2$ is an integer and $H$ is a subgroup of $\Sym_n$. Again  $A$ is a semigroup algebra $K[S_{n,l}(H)]$,
where $$S_{n,l}(H)=\langle x_1,x_2,\dots ,x_n\mid
x_{i_1}x_{i_2}\cdots x_{i_l} =x_{\sigma (i_1)} x_{\sigma (i_2)}
\cdots
 x_{\sigma (i_l)},\; \sigma\in
H\rangle $$ is the monoid with the ``same'' presentation as the
algebra. So, the defining homogeneous relations, in general, do not involve all the defining generators and hence these algebras reflect
more the original motivating examples.

In  Section 2, properties of the monoid $S_{n,l}(H)$ are obtained.
We prove that $S_{n,l}(H)$ is cancellative if and only if $H$ is
semi-regular and abelian. In this case,  $S_{n,l}(H)$ is a submonoid
of its universal group $G$.  When $H$ is a transitve abelian
subgroup of $\Sym_n$, then the periodic elements $T(G)$ of $G$ form
a finite abelian subgroup, $G$ is periodic-by-cyclic and
$G=S_{n,l}(H)\langle x_1^l\rangle^{-1}$.

In Section 3, we consider the algebra $K[S_{n,l}(H)]$. Using the results obtained on the group of
fractions $G$, we prove that for a transitive abelian subgroup $H$ of $\Sym_n$, the Jacobson radical $J(K[G])$ of $K[G]$ is contained in $ J(K[T(G)])K[G]$.
If in addition $\ch (K) = 0$ or $\ch (K) = p$ and $p \nmid |H| $, then $J(K[S_{n,l}(H)])=0$ and $J(K[G])=0$. Furthermore, for arbitrary $H$, the growth of $K[S_{n,l}(H)]$ is
exponential if $H$ is non-transitive, and if $H$ is transitive, then $K[S_{n,l}(H)]$ is a Noetherian PI-algebra with
Gelfand-Kirillov dimension one.  Sufficient conditions for $K[S_{n,l}(H)]$ to be a finitely generated free module
over a finitely generated free $K$-algebra as well as sufficient conditions for $K[G]$ to be a PI-algebra and a principal ideal ring
are given.

\section{Permutation Relations of length $l$}
Let $n$ be a positive integer. Let $l$ be an integer greater than
$1$. Let $H$ be a subgroup of the symmetric group $\Sym_n$. We
define the monoid $S_{n,l}(H)$ with presentation
$$S_{n,l}(H)=\langle x_1,x_2,\dots ,x_n\mid
x_{i_1}x_{i_2}\cdots x_{i_l} =x_{\sigma (i_1)} x_{\sigma (i_2)}
\cdots
 x_{\sigma (i_l)}, \; \sigma \in H,\;  i_1,\ldots ,  i_l \in  \{ 1, \ldots , n \}
\rangle.$$ Since the defining relations of $S_{n,l}(H)$ are
homogeneous, it has a natural degree or length function and for $s
\in S_{n,l}(H)$, its length is denoted by $|s|$.

We begin with an easy consequence of the definition of $S_{n,l}(H)$.

\begin{lemma}\label{central}
Let $H$ be a subgroup of $\Sym_n$. Let $X_i=\{ \sigma(i)\mid
\sigma\in H\}$. If $j\in X_i$, then $x_i^l=x_j^l$ and therefore
$x_i^lx_j=x_jx_i^l$ in $S_{n,l}(H)$. In particular, if $H$ is
transitive, then $x_1^l$ is a central element in $S_{n,l}(H)$.
$\qed$
\end{lemma}

\begin{lemma}\label{form}
Let $H$ be a subgroup of $\Sym_n$. Assume that the different orbits
under the action of $H$ are $X_{i_1},\dots,X_{i_r}$, where $X_i$
denotes the orbit of $i$. Then $F=\langle x_{i_1},\dots
,x_{i_r}\rangle$ is a free submonoid of $S_{n,l}(H)$ with basis
$\{x_{i_1},\dots ,x_{i_r}\}$. Furthermore,  if $s\in S_{n,l}(H)$
 and $|s|\geq l-1$, then there exist $j_1,j_2,\ldots ,
j_{l-1}, k_1, k_2, \ldots , k_{l-1} \in\{ 1,2,\dots ,n\}$ and unique
$w_1,w_2\in F$ such that $s=w_1x_{j_1}x_{j_2} \cdots x_{j_{l-1}}
=x_{k_1} x_{k_2} \cdots x_{k_{l-1}} w_2$.
\end{lemma}

\begin{proof}
By the defining relations of $S_{n,l}(H)$ it is clear that if
$x_{j_1}\cdots x_{j_t}=x_{k_1}\cdots x_{k_u}$ then $t=u$ and $k_p\in
X_{j_p}$, for $p\in \{ 1,2,\dots,t\}$. Therefore $F=\langle
x_{i_1},\dots ,x_{i_r}\rangle$ is a free submonoid of $S_{n,l}(H)$
with basis $\{x_{i_1},\dots ,x_{i_r}\}$.

Let $s=x_{m_1}\cdots x_{m_t}\in S_{n,l}(H)$ such that $t\geq l-1$.
We shall prove the second part of the result by induction on $t$. If
$t=l-1$, then $w_1=w_2=1$ and $s=w_1x_{m_1}\cdots
x_{m_{l-1}}=x_{m_1}\cdots x_{m_{l-1}}w_2$. Suppose that $t>l-1$. By
induction hypothesis there exist $j_1,j_2,\ldots , j_{l-1}, k_1,
k_2, \ldots , k_{l-1} \in\{ 1,2,\dots ,n\}$ and $w_1,w_2\in F$ such
that $s=w_1x_{j_1}x_{j_2}\cdots x_{j_{l-1}}x_{m_t} =x_{m_1} x_{k_1}
x_{k_2} \cdots x_{k_{l-1}} w_2$.  We may assume that $j_1\in
X_{i_p}$ and $k_{l-1}\in X_{i_q}$. Thus there exist $\sigma,\tau\in
H$ such that $\sigma(j_1)=i_p$ and $\tau(k_{l-1})=i_q$. Therefore
\begin{eqnarray*}s&=& w_1x_{j_1}x_{j_2} \cdots x_{j_{l-1}}x_{m_t}=w_1x_{\sigma (j_1)} x_{\sigma (j_2)}\cdots  x_{\sigma (j_{l-1})} x_{\sigma
(m_t)}\\
&=&w_1x_{i_p} x_{\sigma (j_2)}\cdots  x_{\sigma (j_{l-1})}
x_{\sigma (m_t)}\\
s &=& x_{m_1} x_{k_1} x_{k_2} \cdots x_{k_{l-1}} w_2 = x_{\tau(m_1)}
x_{\tau(k_1)} \cdots x_{\tau(k_{l-2})}x_{\tau(k_{l-1})}w_2\\
&=&x_{\tau(m_1)} x_{\tau(k_1)} \cdots x_{\tau(k_{l-2})}x_{i_q}w_2
\end{eqnarray*}
 and $w_1x_{i_p}, x_{i_q}w_2\in F$. The uniqueness of $w_1x_{i_p},
x_{i_q}w_2$ follows easily by the defining relations of
$S_{n,l}(H)$.
\end{proof}

By a sequence of rewrites of $p$ steps, we mean applying a sequence
$\tau_1,\tau_2,\dots, \tau_p \in H$ to the indices  of $l$
consecutive letters of a word $y_{j_1}\dots y_{j_t}$  in the free
monoid $\FM_n$ with basis $\{ y_1,y_2,\dots ,y_n\}$. For example if
$l=2$, if we have the word $w=y_1y_2y_5y_3y_1$ and $\gamma , \delta
, \epsilon , \zeta \in H$, then a possible sequence of rewrites of
$4$ steps would be
$$y_1y_2y_5y_3y_1,\; y_1y_{\gamma (2)}y_{\gamma (5)}y_3y_1,\;
y_1y_{\gamma (2)}y_{\gamma (5)}y_{\delta (3)}y_{\delta (1)}, \;
y_1y_{\gamma (2)}y_{ \epsilon\gamma (5)}y_{\epsilon\delta
(3)}y_{\delta (1)},\; y_1y_{\zeta \gamma (2)}y_{\zeta \epsilon\gamma
(5)}y_{\epsilon\delta (3)}y_{\delta (1)}.$$
Let $\pi\colon
\FM_n\longrightarrow S_{n,l}(H)$ be the unique homomorphism such
that $\pi(y_j)=x_j$, for all $j\in\{ 1,2,\dots, n\}$. Note that
$x_{j_1}x_{j_2}\cdots x_{j_t}=x_{k_1}x_{k_2}\cdots x_{k_t}$ if and
only if there exist a nonnegative integer $p$ and a sequence of
rewrites of $p$ steps $w_1,w_2,\dots, w_{p+1}$ such that
$w_1=y_{j_1}y_{j_2}\cdots y_{j_t}$ and $w_{p+1}=y_{k_1}y_{k_2}\cdots
y_{k_t}$. Furthermore, if $H$ is abelian  and $t\geq l$, then
$x_{j_1}x_{j_2}\cdots x_{j_t}=x_{k_1}x_{k_2}\cdots x_{k_t}$ if and
only if there exists a sequence of rewrites of $t-l+1$ steps of the
form \begin{eqnarray*}&&y_{j_1}y_{j_2}\cdots y_{j_t},\;
y_{\tau_1(j_1)}\cdots y_{\tau_1(j_{l})}y_{j_{l+1}}\cdots y_{j_t},\;
y_{\tau_1(j_1)}y_{\tau_2\tau_1(j_2)}\cdots
y_{\tau_2\tau_1(j_l)}y_{\tau_2(j_{l+1})}y_{j_{l+2}}\cdots
y_{j_{t}},\;\dots ,\\
&&y_{\tau_1(j_1)}\cdots y_{\tau_l\cdots\tau_2\tau_1(j_l)}
y_{\tau_{l+1}\cdots\tau_3\tau_2(j_{l+1})}\cdots
y_{\tau_{t-l+1}\cdots\tau_{t-2l+2}(j_{t-l+1})}y_{\tau_{t-l+1}\cdots\tau_{t-2l+3}(j_{t-l+2})}\cdots
y_{\tau_{t-l+1}\tau_{t-l}(j_{t-1})}y_{\tau_{t-l+1}(j_{t})},\end{eqnarray*}
such that the last word is equal to $y_{k_1}y_{k_2}\cdots y_{k_t}$,
for some $\tau_1,\dots, \tau_{t-l+1}\in H$.

Recall that a subgroup $H$ of $\Sym_n$ is said to be  semi-regular if the
stabilizer of $i$, $H_i=\{\sigma\in H\mid \sigma(i)=i\}$, is trivial
for all $i\in\{ 1,2, \dots, n\}$.

\begin{theorem}\label{cancellative}
Let $H$ be a subgroup of $\Sym_n$. Then $S_{n,l}(H)$ is cancellative
if and only if $H$ is semi-regular and abelian.
\end{theorem}

\begin{proof} Suppose that $H$ is semi-regular and abelian.
Let $a,b,c\in S_{n,l}(H)$ be such that $ab=ac$. We know that $b$ and
$c$ have the same length, say $m$, so we can express the three
elements as $a=x_{j_1}x_{j_2}\cdots x_{j_r}$,
$b=x_{k_1}x_{k_2}\cdots x_{k_m}$ and $c=x_{k'_1}x_{k'_2}\cdots
x_{k'_m}$.  If $r+m<l$, then by the defining relations of
$S_{n,l}(H)$ it is clear that $b=c$. Hence we may assume that
$m+r\geq l$. Since $ab=ac$ and $H$ is abelian, we know that there
exist $\tau_1,\tau_2,\dots ,\tau_{r+m-l+1}\in H$ such that
$y_{j_1}y_{j_2}\cdots y_{j_r}y_{k'_{1}} \cdots
y_{k'_{m-1}}y_{k'_{m}} =$ $$
y_{\tau_1(j_1)}y_{\tau_2\tau_1(j_2)}\cdots
y_{\tau_{r'}\tau_{r'-1}\cdots\tau_{s'}(j_r)}
y_{\tau_{r''}\tau_{r''-1}\cdots\tau_{s''} (k_{1}) }\cdots
y_{\tau_{r+m-l+1}\tau_{r+m-l} (k_{m-1}) }y_{\tau_{r+m-l+1} (k_{m}) }
$$
in $\FM_n$, where
$$(r',s',r'',s'')=\left\{\begin{array}{ll}
(r,1,r+1,1)&\mbox{ if }r<l \mbox{ and
}r<r+m-l+1\\
(r,1,r,1)&\mbox{ if }r< l\mbox{ and
}r=r+m-l+1\\
(r+m-l+1,1,r+m-l+1,1)&\mbox{ if }r<l\mbox{ and }r>r+m-l+1\\
(r,r-l+1,r+1,r-l+2)&\mbox{ if }r\geq l \mbox{ and
}r<r+m-l+1\\
(r,r-l+1,r,r-l+2)&\mbox{ if }r\geq  l\mbox{ and
}r=r+m-l+1\\
(r+m-l+1,r-l+1,r+m-l+1,r-l+2)&\mbox{ if }r\geq l\mbox{ and
}r>r+m-l+1
\end{array}\right.$$
Therefore we have $\tau_1(j_1)=j_1, \tau_2\tau_1(j_2)=j_2, \ldots ,
\tau_{r'}\tau_{r'-1}\cdots\tau_{s'}(j_r)=j_r$. Since $H$ is
semi-regular, this implies that $\tau_1=\tau_2=\ldots =
\tau_{r'}=\id$, and therefore $x_{k_1}x_{k_2}\cdots x_{k_{m-1}}
x_{k_{m}} = x_{k'_1}x_{k'_2}\cdots x_{k'_{m-1}} x_{k'_{m}} $, so
$b=c$. Therefore $S_{n,l}(H)$ is left-cancellative and, by a similar
argument, it is right-cancellative.

Suppose that $S_{n,l}(H)$ is cancellative. Let $\tau\in H_{j_1}$.
Then $x_{j_1}x_{j_2}x_{j_3}\cdots
x_{j_l}=x_{\tau(j_1)}x_{\tau(j_2)}x_{\tau(j_3)}\cdots x_{\tau(j_l)}
=x_{j_1}x_{\tau(j_2)}x_{\tau(j_3)}\cdots x_{\tau(j_l)}$, and since
$S_{n,l}(H)$ is  cancellative,
 $x_{j_2}x_{j_3}\cdots x_{j_l} = x_{\tau(j_2)}x_{\tau(j_3)}\cdots x_{\tau(j_l)}$.
So,
$x_{j_q}=x_{\tau({j_q})}$, for all
${j_q}\in\{ 1,2,\dots ,n\}$. Hence $\tau=\id$ and therefore $H$ is
semi-regular.

Let $\mu , \nu \in H$. We have
\allowdisplaybreaks{
\begin{align*}
 x_rx_{s_1}x_{s_2}\cdots x_{s_{l-1}}x_t&=x_{\mu(r)} x_{\mu({s_1})}x_{\mu({s_2})}\cdots x_{\mu({s_{l-1}})}  x_t =
x_{\mu(r)} x_{\nu\mu({s_1})}x_{\nu\mu({s_2})}\cdots x_{\nu\mu({s_{l-1}})}  x_{\nu(t)} \\
x_rx_{s_1}x_{s_2}\cdots x_{s_{l-1}}x_t&=x_r x_{\nu({s_1})}x_{\nu({s_2})}\cdots x_{\nu({s_{l-1}})} x_{\nu(t)} =
x_{\mu(r)} x_{\mu \nu({s_1})}x_{\mu \nu({s_2})}\cdots x_{\mu \nu({s_{l-1}})} x_{\nu(t)}
\end{align*}}
Since $S_{n,l}(H)$ is cancellative,  $x_{\nu \mu({s_q})}= x_{\mu \nu
({s_q})}$, for all ${s_q} \in \{ 1,2,\dots ,n\}$. Therefore $ \nu \mu= \mu
\nu$ and the result follows.
\end{proof}

\begin{proposition}\label{group}
Let $H$ be a transitive abelian subgroup of $\Sym_n$. Then
$S_{n,l}(H)$ has a group of fractions $G=S_{n,l}(H)\langle
x_1^l\rangle^{-1}$. Furthermore
\begin{itemize}
\item[(i)] $T(G)=\{ g\in G\mid o(g)<\infty\}=\{x_1^{-l+1}x_{j_1}\cdots x_{j_{l-1}}\mid
j_i\in\{ 1,2,\dots ,n \}\}$ is a normal abelian subgroup of $G$, and
$T(G)\cong H^{l-1}$, the direct product of $l-1$ copies of $H$.
\item[(ii)] $G=T(G)\rtimes gr\langle x_1\rangle$
\item[(iii)] $gr\langle x_1^l\rangle$ is a central subgroup of
finite  index in $G$.
\end{itemize}
\end{proposition}
\begin{proof} By \cite[Proposition~3.2]{permutation}, $H$ is
regular, i.e. $H$ is transitive and semi-regular.  By
Theorem~\ref{cancellative}, $S_{n,l}(H)$ is cancellative. By
Lemma~\ref{central}, $x_1^l$ is a central element in $S_{n,l}(H)$.
Since $x_1^l=x_i^l$, for all $i$, it is easy to see that the central
localization $S_{n,l}(H)\langle x_1^l\rangle^{-1}$ is a group of
fractions of $S_{n,l}(H)$, in fact $x_i^{-1}=x_i^{l-1}x_1^{-l}$.
Thus $(iii)$ follows easily by Lemma~\ref{form}. Note that the group
$G=S_{n,l}(H)\langle x_1^l\rangle^{-1}$ is naturally
$\mathbb{Z}$-graded, where the degree of
$x_{i_1}^{\varepsilon_1}x_{i_2}^{\varepsilon_2}\cdots
x_{i_k}^{\varepsilon_k}$ is $\sum_{j=1}^k\varepsilon_j$. Hence the
elements of finite order of $G$ have degree zero. By
Lemma~\ref{form}, the set of elements of degree zero is
$$T=\{x_1^{-l+1}x_{j_1}\cdots x_{j_{l-1}}\mid
j_1,\dots ,j_{l-1}\in\{ 1,2,\dots ,n \}\}.$$ Clearly this set is a
normal subgroup of $G$. Since $H$ is regular, for every $i\in \{
1,2,\dots, n\}$, there exists a unique $\sigma_i\in H$ such that
$\sigma_i(i)=1$. Let $f\colon H^{l-1}\longrightarrow T$ be the
bijective function defined by $f(\sigma_{j_1},\dots
,\sigma_{j_{l-1}})=x_1^{-l+1}x_{j_1}\cdots x_{j_{l-1}}$. Note that
\begin{eqnarray*}
x_1^{-l+1}x_{j_1}\cdots x_{j_{l-1}}x_1^{-l+1}x_{k_1}\cdots
x_{k_{l-1}}&=&x_1^{-2l}x_1x_{j_1}\cdots x_{j_{l-1}}x_1x_{k_1}\cdots
x_{k_{l-1}}\\
&=& x_1^{-2l}x_1^2x_{\sigma_{j_1}(j_2)}\cdots x_{\sigma_{j_1}(j_{l-1})}x_{\sigma_{j_1}(1)}x_{k_1}\cdots x_{k_{l-1}}\\
&=& x_1^{-2l}x_1^3x_{\sigma_{j_2}(j_3)}\cdots x_{\sigma_{j_2}(j_{l-1})}x_{\sigma_{j_2}(1)}x_{\sigma_{j_2}\sigma_{j_1}^{-1}(k_1)}x_{k_2}\cdots x_{k_{l-1}}\\
&=& \ldots =
x_1^{-2l}x_1^lx_{\sigma_{j_{l-1}}(1)}x_{\sigma_{j_{l-1}}\sigma_{j_1}^{-1}(k_1)}x_{\sigma_{j_{l-1}}\sigma_{j_2}^{-1}(k_2)}\cdots
x_{\sigma_{j_{l-1}}\sigma_{j_{l-2}}^{-1}(k_{l-2})}x_{k_{l-1}}\\
&=&x_1^{-2l}x_1^{l+1}x_{\sigma_{j_1}^{-1}(k_1)}x_{\sigma_{j_2}^{-1}(k_2)}\cdots
x_{\sigma_{j_{l-2}}^{-1}(k_{l-2})}x_{\sigma_{j_{l-1}}^{-1}(k_{l-1})}
\end{eqnarray*}
Since $\sigma_{\sigma_j^{-1}(k)}(\sigma_j^{-1}(k))=1=\sigma_k(k)$,
and $H_k$ is trivial, we have that
$\sigma_{\sigma_j^{-1}(k)}\sigma_j^{-1}=\sigma_k$. Because $H$ is
abelian, $\sigma_{\sigma_j^{-1}(k)}=\sigma_j\sigma_k$. Therefore
$f((\sigma_{j_1},\dots ,\sigma_{j_{l-1}})(\sigma_{k_1},\dots
,\sigma_{k_{l-1}}))=f(\sigma_{j_1},\dots
,\sigma_{j_{l-1}})f(\sigma_{k_1},\dots ,\sigma_{k_{l-1}})$ and thus  $(i)$ follows. By Lemma~\ref{form} and since $T(G)\cap
gr\langle x_1\rangle=\{ 1\}$, $(ii)$ follows easily.
\end{proof}

\begin{theorem}\label{universalgroupThree}
Let $H$ be an abelian semi-regular subgroup of $\Sym_n$. Then
$S_{n,l}(H)$ is a submonoid of its universal group
$$G=\gr(x_1,\dots, x_n\mid
x_{i_1}x_{i_2}\cdots x_{i_l}=x_{\sigma(i_1)}x_{\sigma(i_2)}\cdots
x_{\sigma(i_l)}, \mbox{ for all }i_1,i_2,\dots ,i_l \in\{1,2,\dots
,n\} \mbox{ and }\sigma\in H).$$
\end{theorem}

\begin{proof}
Assume that the different orbits under the action of $H$ are
$X_{i_1},\dots,X_{i_r}$, where $X_i$ denotes the orbit of $i$.
Consider the free group $F$ on $\{ x_{i_1},x_{i_2},\dots
,x_{i_r}\}$. We know that every element $w\in F$ can be uniquely
written in the form
\begin{eqnarray}\label{rrreduced}
&&w=x_{i_{j_1}}^{\varepsilon_1}x_{i_{j_2}}^{\varepsilon_2}\cdots
x_{i_{j_m}}^{\varepsilon_m},
\end{eqnarray}
for some nonnegative integer $m$, elements $j_1,\dots, j_m\in \{
1,2,\dots ,r\}$ and $\varepsilon_1,\dots,\varepsilon_m\in \{
-1,1\}$, such that if $m>1$ and $j_t=j_{t+1}$, for some $1\leq t<m$,
then $\varepsilon_t=\varepsilon_{t+1}$. We say that the expression
on the right hand side of (\ref{rrreduced}) is the normal form of $w$
(when $m=0$ then  (\ref{rrreduced}) is $w=1$). We say that the degree of
$w$ (as in (\ref{rrreduced})) is
$\deg(w)=\sum_{i=1}^m\varepsilon_i$.

Let $\mathcal{C}=\{ (w,x_{\sigma_1 (i_{1})}\cdots x_{\sigma_{l-1}
(i_{1})})\mid w\in F,\,\sigma_j \in H$ for $j=1,\dots ,l-1\}$. For
$j\in\{ 1,2,\dots ,r\}$ and $\tau\in H$, we define $f_{\tau,
j}\colon \mathcal{C}\longrightarrow \mathcal{C}$ by
$$f_{\tau, j}(w,x_{\sigma_1 (i_{1})}\cdots x_{\sigma_{l-1}
(i_{1})})=\left\{\begin{array}{l}
(x_{i_j}w,x_{\tau^{-1}\sigma_1(i_1)} x_{\tau^{-1}\sigma_2(i_1)}
\cdots x_{\tau^{-1}\sigma_{l-1} (i_{1})})\quad \mbox{ if }
\deg(w)\equiv
0\, (\mathrm{mod}(l))\\
(x_{i_j}w,x_{\sigma_1(i_1)}\cdots
x_{\sigma_{k-1}(i_1)}x_{\tau\sigma_{k} (i_{1})}x_{\sigma_{k+1}
(i_{1})}\cdots x_{\sigma_{l-1} (i_{1})})\vspace{5pt}\\
\hphantom{xxxxxxxxxxxxxxxxxxxxxxxxxxxxixxx} \mbox{ if }
\deg(w)\equiv k\, (\mathrm{mod}(l))\end{array}\right.$$ for all
$(w,x_{\sigma_1 (i_{1})}\cdots x_{\sigma_{l-1} (i_{1})})\in
\mathcal{C}$. It is clear that $f_{\tau ,j}$ is bijective, in fact
$$f_{\tau, j}^{-1}(w,x_{\sigma_1 (i_{1})}\cdots x_{\sigma_{l-1}
(i_{1})})=\left\{\begin{array}{l}
(x_{i_j}^{-1}w,x_{\tau\sigma_1(i_1)} x_{\tau\sigma_2(i_1)} \cdots
x_{\tau\sigma_{l-1} (i_{1})})\quad \mbox{ if } \deg(w)\equiv
1\, (\mathrm{mod}(l))\\
(x_{i_j}^{-1}w,x_{\sigma_1(i_1)}\cdots
x_{\sigma_{k-1}(i_1)}x_{\tau^{-1}\sigma_{k} (i_{1})}x_{\sigma_{k+1}
(i_{1})}\cdots x_{\sigma_{l-1} (i_{1})})\vspace{5pt}\\
\hphantom{xxxxxxxxxxxxxxxxxxxxxxxixxxx} \mbox{ if } \deg(w)\equiv
k+1\, (\mathrm{mod}(l))\end{array}\right.$$ for all $(w,x_{\sigma_1
(i_{1})}\cdots x_{\sigma_{l-1} (i_{1})})\in \mathcal{C}$.

Let $\nu_1,\dots ,\nu_l,\tau\in H$ and $j_1,\dots, j_l\in \{
1,2,\dots ,r\}$. We shall prove that
$$f_{\nu_1,j_1} \circ \cdots \circ f_{\nu_l,j_l}=
f_{\tau\nu_1,j_1} \circ \cdots \circ f_{\tau\nu_l,j_l}.$$ For
$1<k< l$ we have that $(f_{\nu_1,j_1} \circ \cdots \circ
f_{\nu_l,j_l})(w,x_{\sigma_1 (i_{1})}\cdots x_{\sigma_{l-1}
(i_{1})})$
$$=(x_{i_{j_1}}\cdots
x_{i_{j_l}}w,x_{\nu_k^{-1}\nu_{k-1}\sigma_1 (i_{1})}\cdots
x_{\nu_k^{-1}\nu_{1}\sigma_{k-1}
(i_{1})}x_{\nu_k^{-1}\nu_{l}\sigma_k
(i_{1})}x_{\nu_k^{-1}\nu_{l-1}\sigma_{k+1} (i_{1})}\cdots
x_{\nu_k^{-1}\nu_{k+1}\sigma_{l-1} (i_{1})})$$ if $\deg(w)\equiv k
\, (\mathrm{mod}(l))$. We also have that $(f_{\tau\nu_1,j_1} \circ
\cdots \circ f_{\tau\nu_l,j_l})(w,x_{\sigma_1 (i_{1})}\cdots
x_{\sigma_{l-1} (i_{1})})$
$$=(x_{i_{j_1}}\cdots
x_{i_{j_l}}w,x_{\nu_k^{-1}\nu_{k-1}\sigma_1 (i_{1})}\cdots
x_{\nu_k^{-1}\nu_{1}\sigma_{k-1}
(i_{1})}x_{\nu_k^{-1}\nu_{l}\sigma_k
(i_{1})}x_{\nu_k^{-1}\nu_{l-1}\sigma_{k+1} (i_{1})}\cdots
x_{\nu_k^{-1}\nu_{k+1}\sigma_{l-1} (i_{1})})$$ if $\deg(w)\equiv k
\, (\mathrm{mod}(l))$.

For $k=1$ we have that $(f_{\nu_1,j_1} \circ \cdots \circ
f_{\nu_l,j_l})(w,x_{\sigma_1 (i_{1})}\cdots x_{\sigma_{l-1}
(i_{1})})$
$$=(x_{i_{j_1}}\cdots
x_{i_{j_l}}w,x_{\nu_1^{-1}\nu_{l}\sigma_1
(i_{1})}x_{\nu_1^{-1}\nu_{l-1}\sigma_{2} (i_{1})}\cdots
x_{\nu_1^{-1}\nu_{2}\sigma_{l-1} (i_{1})})$$ if $\deg(w)\equiv 1 \,
(\mathrm{mod}(l))$, and $(f_{\tau\nu_1,j_1} \circ \cdots \circ
f_{\tau\nu_l,j_l})(w,x_{\sigma_1 (i_{1})}\cdots x_{\sigma_{l-1}
(i_{1})})$
$$=(x_{i_{j_1}}\cdots
x_{i_{j_l}}w,x_{\nu_1^{-1}\nu_{l}\sigma_1
(i_{1})}x_{\nu_1^{-1}\nu_{l-1}\sigma_{2} (i_{1})}\cdots
x_{\nu_1^{-1}\nu_{2}\sigma_{l-1} (i_{1})})$$ if $\deg(w)\equiv 1 \,
(\mathrm{mod}(l))$.

 For $k=l$ we have that $(f_{\nu_1,j_1} \circ \cdots \circ
f_{\nu_l,j_l})(w,x_{\sigma_1 (i_{1})}\cdots x_{\sigma_{l-1}
(i_{1})})$
$$=(x_{i_{j_1}}\cdots
x_{i_{j_l}}w,x_{\nu_l^{-1}\nu_{l-1}\sigma_1 (i_{1})}\cdots
x_{\nu_l^{-1}\nu_{1}\sigma_{l-1} (i_{1})})$$ if $\deg(w)\equiv l \,
(\mathrm{mod}(l))$, and $(f_{\tau\nu_1,j_1} \circ \cdots \circ
f_{\tau\nu_l,j_l})(w,x_{\sigma_1 (i_{1})}\cdots x_{\sigma_{l-1}
(i_{1})})$
$$=(x_{i_{j_1}}\cdots
x_{i_{j_l}}w,x_{\nu_l^{-1}\nu_{l-1}\sigma_1 (i_{1})}\cdots
x_{\nu_l^{-1}\nu_{1}\sigma_{l-1} (i_{1})})$$ if $\deg(w)\equiv l \,
(\mathrm{mod}(l))$. Hence
$$f_{\nu_1,j_1} \circ \cdots \circ f_{\nu_l,j_l}=
f_{\tau\nu_1,j_1} \circ \cdots \circ f_{\tau\nu_l,j_l}.$$

Therefore there exists a unique homomorphism $\varphi\colon
S_{n,l}(H)\longrightarrow \Sym(\mathcal{C})$ such that
$\varphi(x_{\nu(i_j)})=f_{\nu, j}$, for all $\nu\in H$ and $j\in\{
1,2,\dots, r\}$.  It remains to show that  $\varphi$ is injective.
Indeed, as  $S_{n,l}(H)$ is a submonoid of a group it then follows
that  it is a submononoid of its universal group. So, assume $a,b\in
S_{n,l}(H)$ are such that $\varphi(a)=\varphi(b)$. Without loss of
generality, we may  suppose  that $a\neq 1$. Then, by
Lemma~\ref{form}, either $|a|<l-1$ or there exists a unique $w\in
\widetilde{F}$, where $\widetilde{F}=\langle
x_{i_1},\dots,x_{i_r}\rangle$ is a free submonoid on $\{
x_{i_1},\dots,x_{i_r}\}$ of $S_{n,l}(H)$, and there exist $k_1, k_2,
\ldots ,k_{l-1} \in \{ 1,2,\dots, n\}$ such that
$a=wx_{k_1}x_{k_2}\cdots x_{k_{l-1}}$. Suppose first that $|a|\geq
l-1$. By Theorem~\ref{cancellative}, the sequence $k_1, k_2, \ldots
,k_{l-1}$ is also unique. Since $H$ is semi-regular there exist
$\mu_1,\mu_2, \ldots , \mu_{l-1} \in H$ and $j_1 , j_2, \ldots
,j_{l-1} \in\{ 1,2,\dots ,r\}$ such that $k_1 =\mu_1(i_{j_1}) , k_2=
\mu_2(i_{j_2}) , \ldots , k_{l-1}  = \mu_{l-1}(i_{j_{l-1}})$, and
the sequences $\mu_1,\mu_2, \ldots , \mu_{l-1}$ and $j_1 , j_2,
\ldots ,j_{l-1}$ are unique. Suppose that
$$w=x_{i_{s_1}}\cdots x_{i_{s_q}},$$
for some $s_1,\dots ,s_q\in\{ 1,2,\dots ,r\}$. Then
$\varphi(a)=f_{\id,s_1}\circ\cdots\circ f_{\id,s_q} \circ
f_{\mu_1,j_1}\circ  \cdots \circ f_{\mu_{l-1},j_{l-1}} $. Hence
$\varphi(a)(x_{i_1},x_{i_1}^{l-1})=(wx_{i_{j_1}}\cdots
x_{i_{j_{l-1}}}x_{i_1}, x_{\mu_{l-1}(i_1)}\cdots x_{\mu_{1}(i_1)})$.
Therefore $b\neq 1$ and since
$\varphi(b)(x_{i_1},x_{i_1}^{l-1})=\varphi(a)(x_{i_1},x_{i_1}^{l-1})$,
it is clear that $b=wx_{\mu_1(i_{j_1})} \cdots
x_{\mu_{l-1}(i_{j_{l-1}}) }=a$.

Suppose that $|a|<l-1$. In this case, there exist unique $k_1,\dots,
k_p\in \{ 1,2,\dots ,n\}$ such that $a=x_{k_1}x_{k_2}\cdots
x_{k_p}$. Since $H$ is semi-regular there exist unique $\mu_1,\mu_2,
\ldots , \mu_{p} \in H$ and $j_1 , j_2, \ldots ,j_{p} \in\{
1,2,\dots ,r\}$ such that $k_1 =\mu_1(i_{j_1}) , k_2= \mu_2(i_{j_2})
, \ldots , k_{p}  = \mu_{p}(i_{j_{p}})$. Then
$\varphi(a)=f_{\mu_1,j_1}\circ \cdots \circ f_{\mu_{p},j_{p}} $.
Hence $\varphi(a)(x_{i_1},x_{i_1}^{l-1})=(x_{i_{j_1}}\cdots
x_{i_{j_{p}}}x_{i_1}, x_{\mu_{p}(i_1)}\cdots
x_{\mu_{1}(i_1)}x_{i_1}^{l-1-p})$. Therefore $b\neq 1$ and since
$\varphi(b)(x_{i_1},x_{i_1}^{l-1})=\varphi(a)(x_{i_1},x_{i_1}^{l-1})$,
it is clear that $b=x_{\mu_1(i_{j_1})} \cdots x_{\mu_{p}(i_{j_{p}})
}=a$. Thus indeed $\varphi$ is injective, and the result follows.
\end{proof}

\section{The algebra $K[S_{n,l}(H)]$}
Let $H$ be a subgroup of $\Sym_n$.  For a field $K$, we will study
the algebraic structure of the semigroup algebra $K[S_{n,l}(H)]$.

\begin{theorem}\label{T1}
Let $H$ be a transitive abelian subgroup of $\Sym_n$. Let $G$ be the
group of fractions of $S=S_{n,l}(H)$. Then $J(K[G])\subseteq
J(K[T(G)])K[G]$, in particular $J(K[G])$ is a nilpotent ideal.
Furthermore if $\ch (K) = 0$ or $\ch (K) = p$ and $p \nmid |H| $,
then $J(K[S])=0$ and $J(K[G])=0$.
\end{theorem}
\begin{proof}
The first part of the result follows by Proposition~\ref{group} and
\cite[Theorem~7.3.1]{passman}. Therefore, since $H^{l-1}\cong T(G)$,
by Maschke's Theorem, if $\ch (K) = 0$ or $\ch (K) = p$ and $p \nmid
|H| $, then $J(K[T(G)])=0$. Since $J(K[G])\subseteq J(K[T(G)])K[G]$,
we have $J(K[G])=0$ in this case. Note that $K[S]\subseteq K[G]$.

Suppose that $\ch (K) = 0$ or $\ch (K) = p$ and $p \nmid |H| $. Let
$\alpha \in J(K[S])$ be a homogeneous element (with respect to the
gradation defined by the natural length function on $S$). Then
$\alpha$ is nilpotent (see for example
\cite[Theorem~22.6]{passmancrossed}). We assume that $\alpha \neq
0$. Since $S$ is cancellative, multiplying by  $x_1^{l}$, we may
assume that  the degree of $\alpha$ is greater than $l-1$ and
according to the normal form obtained in the previous section we can
write any element in the support of $\alpha$ as $x_1^mx_{i_1}\cdots
x_{i_{l-1}}$ for some $i_j \in \{1, \ldots , n\}$. Let $I=\{
1,2,\dots ,n\}$. Then $\alpha = \sum_{(i_1,\dots ,i_{l-1})\in
I^{l-1}} \alpha_{i_1,\dots ,i_{l-1}} x_1^m x_{i_1}\cdots
x_{i_{l-1}}$. Since $S$ is cancellative, multiplying by a power of
$x_1$, we may assume that $m$ is a multiple of $l$. Suppose that
$\alpha_{j_1,\dots ,j_{l-1}} \neq 0$. Then \allowdisplaybreaks{
\begin{align*}
\alpha x_{j_{l-1}}^{l-1}\cdots x_{j_1}^{l-1} &= \sum_{(i_1,\dots
,i_{l-1})\in I^{l-1}} \alpha_{i_1,\dots ,i_{l-1}} x_1^m
x_{i_1}\cdots x_{i_{l-1}}  x_{j_{l-1}}^{l-1}\cdots x_{j_1}^{l-1}\\
&=\alpha_{j_1,\dots ,j_{l-1}} x_1^{m+l(l-1)} + \sum_{ (i_1,\dots
,i_{l-1})\in I^{l-1}\setminus\{(j_1,\dots ,j_{l-1})\}}
\alpha_{i_1,\dots ,i_{l-1}} x_1^m x_{i_1}\cdots
x_{i_{l-1}}x_{j_{l-1}}^{l-1}\cdots x_{j_1}^{l-1}
\end{align*}}
Thus $x_1^{m+l(l-1)}\in \supp (\alpha x_{j_{l-1}}^{l-1}\cdots
x_{j_1}^{l-1} )$.  For $\beta = \sum_{g \in G} \beta_g g $, the
trace map is defined by $ \tr \beta = \tr \left( \sum_{g \in G}
\beta_g g \right)= \beta_1$.  Since $\alpha x_{j_{l-1}}^{l-1}\cdots
x_{j_1}^{l-1}\in J(K[S])$ is a homogeneous element, we know that it
is nilpotent. Since $x_1^l$ is central, $\alpha
x_{j_{l-1}}^{l-1}\cdots x_{j_1}^{l-1} x_1^{-(m+l(l-1))}$ is a
nilpotent element in $K[G]$. Clearly,
$$\tr
\left( \alpha x_{j_{l-1}}^{l-1}\cdots x_{j_1}^{l-1}
x_1^{-(m+l(l-1))} \right) =  \alpha_{j_1,\dots ,j_{l-1}}.$$ As by
assumption $\alpha_{j_1,\dots ,j_{l-1}} \neq 0$, this is in
contradiction with \cite[Lemma~2.3.3]{passman}.
 Since $J(K[S])$ is a
homogeneous ideal by \cite[Theorem~22.6]{passmancrossed}, we have
$J(K[S])=0$.
\end{proof}

\begin{theorem}\label{T2}
Let $H$ be a transitive subgroup of $\Sym_n$. Let $S=S_{n,l}(H)$.
Let $K$ be a field. Then $K[S]$ is a noetherian PI-algebra and
$\GKdim(K[S])=1=\clKdim(K[S]$. Furthermore, if $H$ also is abelian
and has no elements of order $\ch(K)$, and $G$ is the group of
fractions of $S$, then $K[G]$ is a PI-algebra and it is a principal
right (and left) ideal ring.
\end{theorem}

\begin{proof} Note that $K[S]$ is a finitely generated right module
over $K[\langle x_1^l\rangle]$. Thus the first part of the result
follows. The second part is a consequence of Proposition~\ref{group}
and \cite[Theorem~4.1]{passman1}
\end{proof}

\begin{theorem}\label{T3} Let $H$ be a non-transitive subgroup of
$\Sym_n$. Let $S=S_{n,l}(H)$. Let $K$ be a field. Then $K[S]$ has
exponential growth. Furthermore, if $H$ is also semi-regular and
abelian, then $K[S]$ is finitely generated free right (left) module
over a finitely generated free $K$-algebra.
\end{theorem}

\begin{proof} Assume that the
different orbits under the action of $H$ are
$X_{i_1},\dots,X_{i_r}$. Since $H$ is not transitive, $r>1$. By
Lemma~\ref{form}, $K[S]$ is a finitely generated right (left) module
over the free subalgebra $K\langle x_{i_1},\dots ,x_{i_r} \rangle$
of rank $r$. Thus the growth of $K[S]$ is exponential.

Suppose that $H$ is also semi-regular and abelian. Let $R=K\langle
x_{i_1},\dots ,x_{i_r} \rangle$. Let $I=\{ 1,\dots ,n\}$. By
Lemma~\ref{form},
$$K[S]=R+\sum_{j=1}^r\sum_{\sigma\in H\setminus\{\id\}}x_{\sigma(i_j)}R
+\sum_{k=1}^{l-2}\sum_{(j_1,\dots ,j_{k})\in
I^k}\sum_{j=1}^r\sum_{\sigma\in H\setminus\{\id\}}x_{j_1}\cdots
x_{j_k}x_{\sigma(i_j)}R.$$ By Theorem~\ref{cancellative}, $S$ is a
cancellative monoid. Therefore it is easy to see that the above sum
is a direct sum of right $R$-modules isomorphic to $R$. Similarly we
have a decomposition of $K[S]$ as a direct sum of left $R$-modules
isomorphic to $R$.
\end{proof}

\vspace{30pt}
 \noindent \begin{tabular}{llllllll}
 F. Ced\'o && E. Jespers  \\
 Departament de Matem\`atiques &&  Department of Mathematics \\
 Universitat Aut\`onoma de Barcelona &&  Vrije Universiteit Brussel  \\
08193 Bellaterra (Barcelona), Spain    && Pleinlaan
2, 1050 Brussel, Belgium \\
 cedo@mat.uab.cat && efjesper@vub.ac.be\\
   &&   \\
G. Klein &&  \\ Department of Mathematics &&
\\  Vrije Universiteit Brussel && \\
Pleinlaan 2, 1050 Brussel, Belgium &&\\ gklein@vub.ac.be&&
\end{tabular}
\end{document}